\documentclass[reqno]{amsart}

\usepackage{amsmath,amssymb,amsthm}

\usepackage{tikz}
\usetikzlibrary{calc,intersections,through,backgrounds}
\usepackage[latin1]{inputenc}

\usepackage{caption}
\usepackage{graphicx}
\usepackage{graphics}

\newtheorem{theorem}{Theorem}
\newtheorem*{thm}{Theorem}
\newtheorem*{lemma}{Lemma}

\newtheorem*{proposition}{Proposition}

\begin{document}

\title[]{On Combinatorial Properties of \\Greedy Wasserstein Minimization}

\author[]{Stefan Steinerberger}
\address{Department of Mathematics, University of Washington, Seattle, WA 98195, USA} \email{steinerb@uw.edu}

\subjclass[2010]{} 
\thanks{S.S. is supported by the NSF (DMS-2123224) and the Alfred P. Sloan Foundation.}

\begin{abstract} We discuss a phenomenon where Optimal Transport leads to a \textit{remarkable} amount of combinatorial regularity. Consider infinite sequences $(x_k)_{k=1}^{\infty}$ in $[0,1]$ constructed in a greedy manner: given $x_1, \dots, x_n$, the new point $x_{n+1}$
is chosen so as to minimize the Wasserstein distance $W_2$ between the empirical measure of the $n+1$ points and the Lebesgue
measure,
$$x_{n+1} = \arg\min_x ~W_2\left( \frac{1}{n+1} \sum_{k=1}^{n} \delta_{x_k} + \frac{\delta_{x}}{n+1},  dx\right).$$
This leads to fascinating sequences (for example: $x_{n+1} = (2k+1)/(2n+2)$ for some $k \in \mathbb{Z}$) which coincide with sequences recently introduced by Ralph Kritzinger in a different setting. Numerically, the regularity of these sequences rival the best known
constructions from Combinatorics or Number Theory. We prove a regularity result below the square root barrier.
\end{abstract}
\maketitle

\section{Introduction and Main Results}
\subsection{Motivation.}   
Given any finite set of points $x_1, \dots, x_n \in [0,1]$, we propose adding a new point in such a way that the Wasserstein $W_2-$distance between the new empirical measure and the Lebesgue measure on $[0,1]$ is minimized. Formally, we consider a notion of energy $E:[0,1] \rightarrow \mathbb{R}_{\geq 0}$ via
$$ E(x) =W_2\left( \frac{1}{n+1} \sum_{k=1}^{n} \delta_{x_k} + \frac{\delta_{x}}{n+1},  dx\right)$$
and then choose $x_{n+1}$ to be any number in which $E$ assumes a global minimum.
Throughout this paper, we will only consider the $W_2$ distance between a discrete measure and the uniform measure on $[0,1]$ for which
the following convenient formula exists: assuming $0 \leq a_1 \leq a_2 \leq \dots \leq a_n \leq 1$, then
$$ W_2\left(\frac{1}{n} \sum_{k=1}^{n} \delta_{a_k}, dx \right)^2 =   \sum_{i=1}^{n} \int_{(i-1)/n}^{i/n} \left( x - a_i\right)^2 dx.$$

This procedure is perhaps best illustrated with an example. 
Consider the initial set of points $\left\{1/\pi, 1/e, 1/\sqrt{2}\right\}$ on $[0,1]$. A computation shows that the point $x_4$ that one needs to add to minimize the $W_2-$transport cost is $x_4 = 7/8$. Repeating the computation after having added $x_4$, the optimal point in the next step is $x_5 = 1/10$. Continuing the computation, one arrives at the sequence
$$ \frac{1}{\pi}, \frac{1}{e}, \frac{1}{\sqrt{2}}, \frac{7}{8}, \frac{1}{10}, \frac{7}{12}, \frac{7}{14}, \frac{13}{16}, \frac{3}{18}, \dots$$
which has a number of remarkable properties. The elements are rational (with numerator odd and denominator being $2n$, see \cite{kritzinger} or \S 2.1). They are uniformly distributed in
the unit interval in a highly regular manner that is illustrated in Fig. 1: newly added points avoid existing points and fill in existing
gaps -- considering the construction, this is not too surprising. However, they seem to do so at the optimal rate which is very difficult to do (see \S 1.3).

\begin{center}
\begin{figure}[h!]
\begin{tikzpicture}
\node at (0,0) {\includegraphics[width=0.45\textwidth]{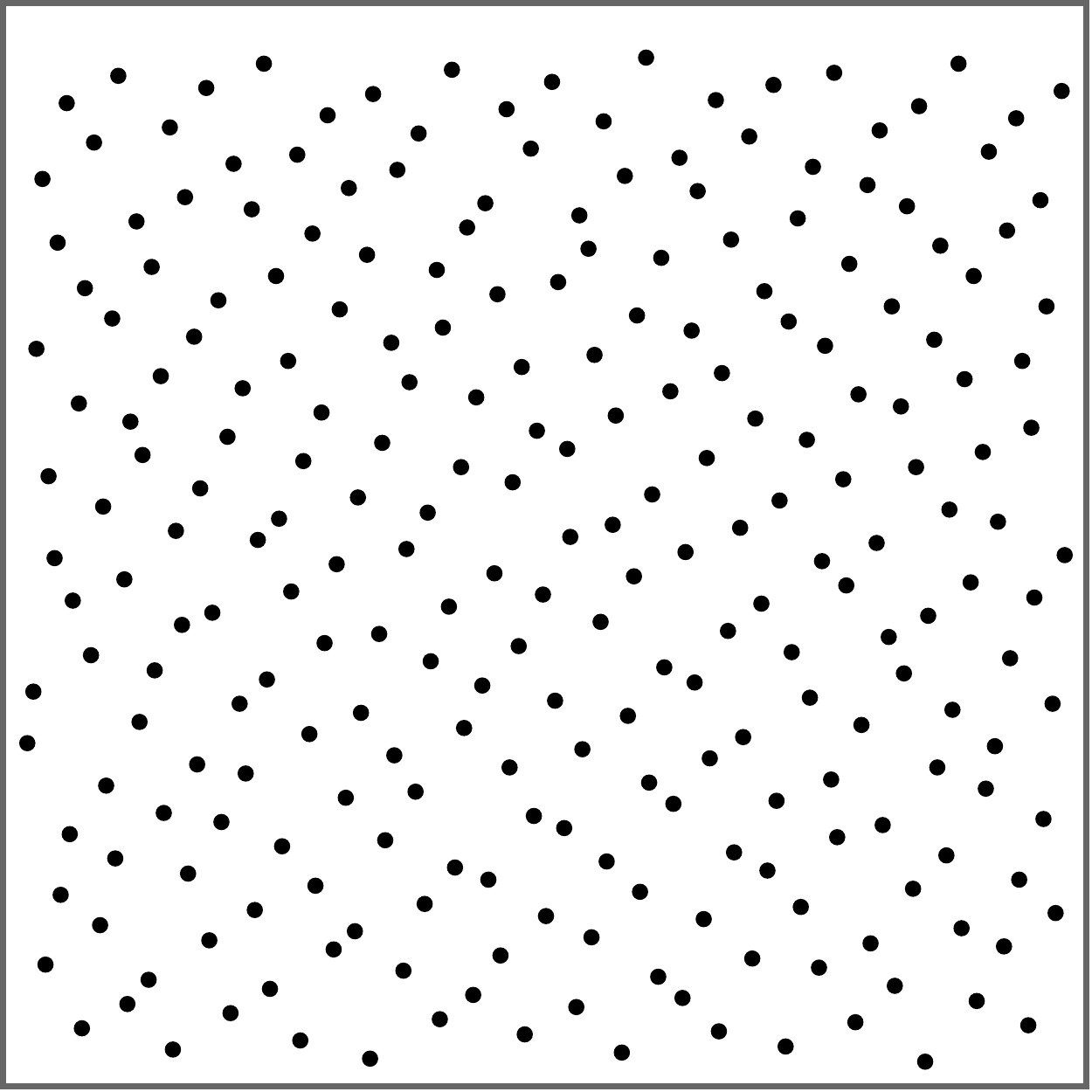}};
\node at (6.2,0) {\includegraphics[width=0.45\textwidth]{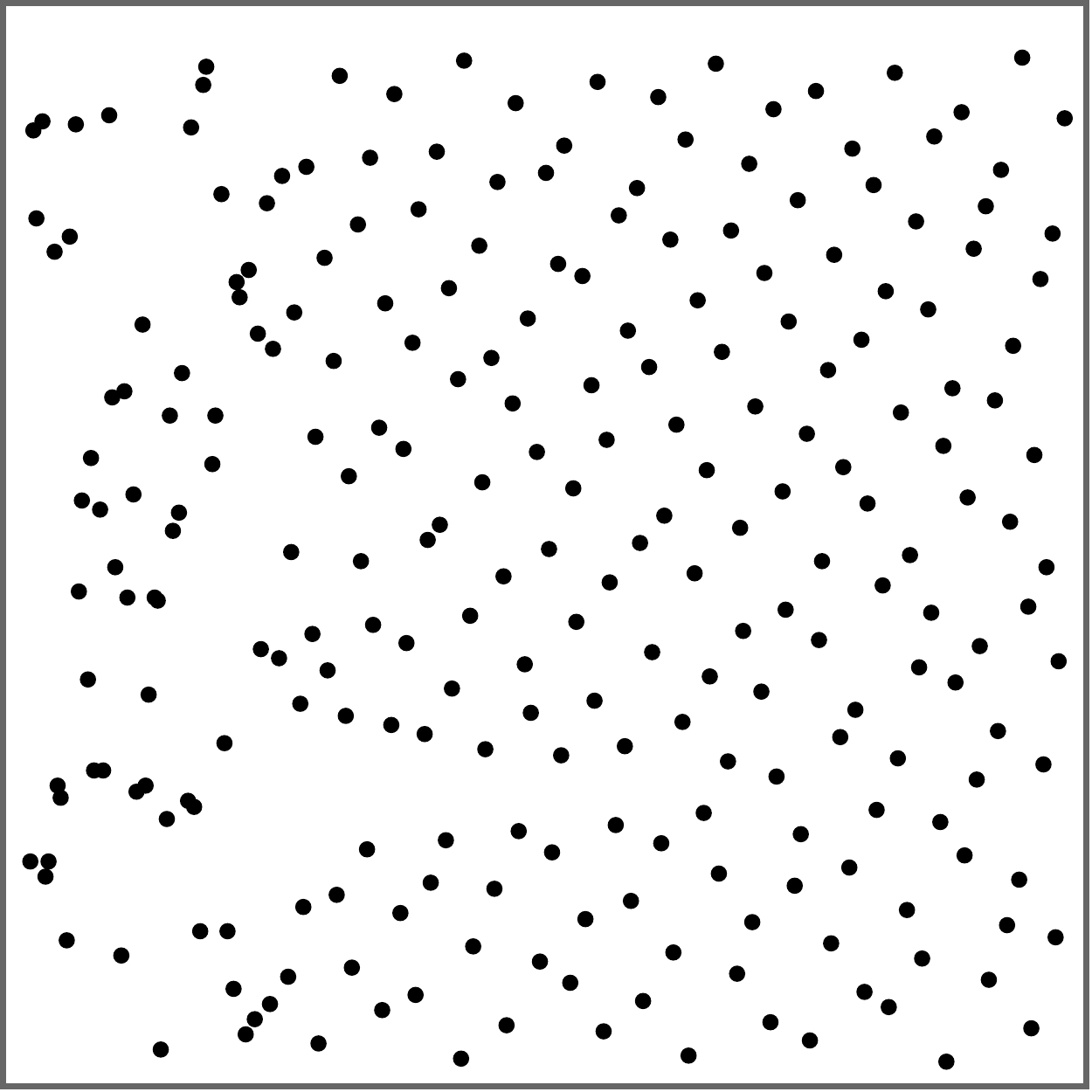}};
\end{tikzpicture}
\caption{An illustration of the first $n$ elements of a sequence: the points $(k/n, ~x_k)$ for $1 \leq k \leq n$ (here $n=250$). Left: starting with $1/\pi, 1/e, 1/\sqrt{2}$. Right: starting with 50 random elements.}
\end{figure}
\end{center}
\vspace{-0pt}
We consider another example (see Fig. 1, right) where we start with 50 randomly sampled points in $[0,1]$. The greedy construction is eager to avoid existing clusters while filling in existing gaps -- what is totally remarkable is that after the greedy construction has added another 50 elements, it seems
to have completely repaired the existing defects and continues in a most regular fashion. 

\subsection{Kritzinger sequences.}
This coincides with another type of construction that was recently given by Ralph Kritzinger \cite{kritzinger} using a seemingly different metric.

\begin{theorem}
Greedy $W_2$ minimization produces a Kritzinger sequence.
\end{theorem}

Kritzinger was motivated by a different notion of regularity which is a classical object in the irregularities of distribution: the $L^2-$discrepancy. Formally, Kritzinger proposes, given $x_1, \dots, x_n$ to pick $x_{n+1}$ so that it minimizes
$$ F(x) = -2 \sum_{k=1}^{n} \max\left\{x_k, x\right\} + (n+1)^2 x^2 - x.$$
This expression is at first glance perhaps a bit difficult to interpret while the Wasserstein $W_2$ interpretation allows access
to the theory of Optimal Transport. Theorem 1 is only true on the unit interval (due to the rigidity of Optimal Transport in one dimension) and fails on domains in $d\geq 2$ dimensions where $W_2$ minimization and Kritzinger sequences are different. We first collect some of the results from \cite{kritzinger}.

\begin{thm}[Kritzinger \cite{kritzinger}]
If $x_n$ is constructed using the greedy method, then it is different from $x_1, \dots, x_{n-1}$ and $x_n = (2k+1)/(2n)$ for some $k \in \mathbb{N}_{\geq 0}$. Moreover, any Kritzinger sequence satisfies, for all $n \in \mathbb{N}$,
$$ \int_{0}^{1} \left| \# \left\{1 \leq k \leq n: x_k \leq x\right\} - n x\right|^2 dx \leq \frac{n}{3} + c,$$
where the constant $c$ only depends on the initial set of points.
\end{thm}

 We give a new proof of this result in our framework before
going on to further refining the estimate. The estimate essentially says that for \textit{most} points $0 < x < 1$
$$ \# \left\{1 \leq k \leq n: 0\leq x_k \leq x\right\} = n \cdot x \pm \mathcal{O}(\sqrt{n}).$$
This is the typical square root law that one would also expect from random variables: it is a common theme in the construction
of greedy sequences that one can usually establish a $\mathcal{O}(\sqrt{n})$ error bound using an averaging trick. Simultaneously,
it seems difficult to go past it in most cases.

\subsection{How regular are these sequences really?}
Suppose now that $(x_k)_{k=1}^{\infty}$ is an arbitrary sequence in $[0,1]$. In 1935 van der Corput \cite{vdc, vdc2} asked whether there exists an infinite sequence so that for some universal constant $C$ and all $n \in \mathbb{N}$
$$ \max_{0 \leq x \leq 1} \left| \# \left\{1 \leq k \leq n: x_k \leq x\right\} -  n x \right| \leq C.$$
This was shown to be impossible by van Aardenne-Ehrenfest \cite{aard} and started the theory of \textit{irregularities of distribution}. 
After seminal contributions by K. F. Roth \cite{roth}, it was then proven by W. M. Schmidt \cite{sch} in 1972 that for \textit{every} sequence
there are infinitely many integers $n \in \mathbb{N}$ such that
$$ \max_{0 \leq x \leq 1} \left| \# \left\{1 \leq k \leq n: x_k \leq x\right\} -  n x \right| \geq c \log{n}.$$
The optimal constant is now known to satisfy $0.065 \leq c \leq 0.2223$ where the lower bound is due to Larcher \& Puchhammer \cite{larch}
and the upper bound is due to Ostromoukhov \cite{ost}. 
There are two classical types of sequences with very good behavior: the Kronecker sequences, also known as irrational rotations on the torus, $x_n = n \alpha ~\mod 1$. The other classical example is the van der Corput sequence in base 2 defined via digit expansions in base 2 (see  \cite{chazelle, dick, drmota, kuipers} for details).
\begin{center}
\begin{figure}[h!]
\begin{tikzpicture}
\node at (0,0) {\includegraphics[width=0.5\textwidth]{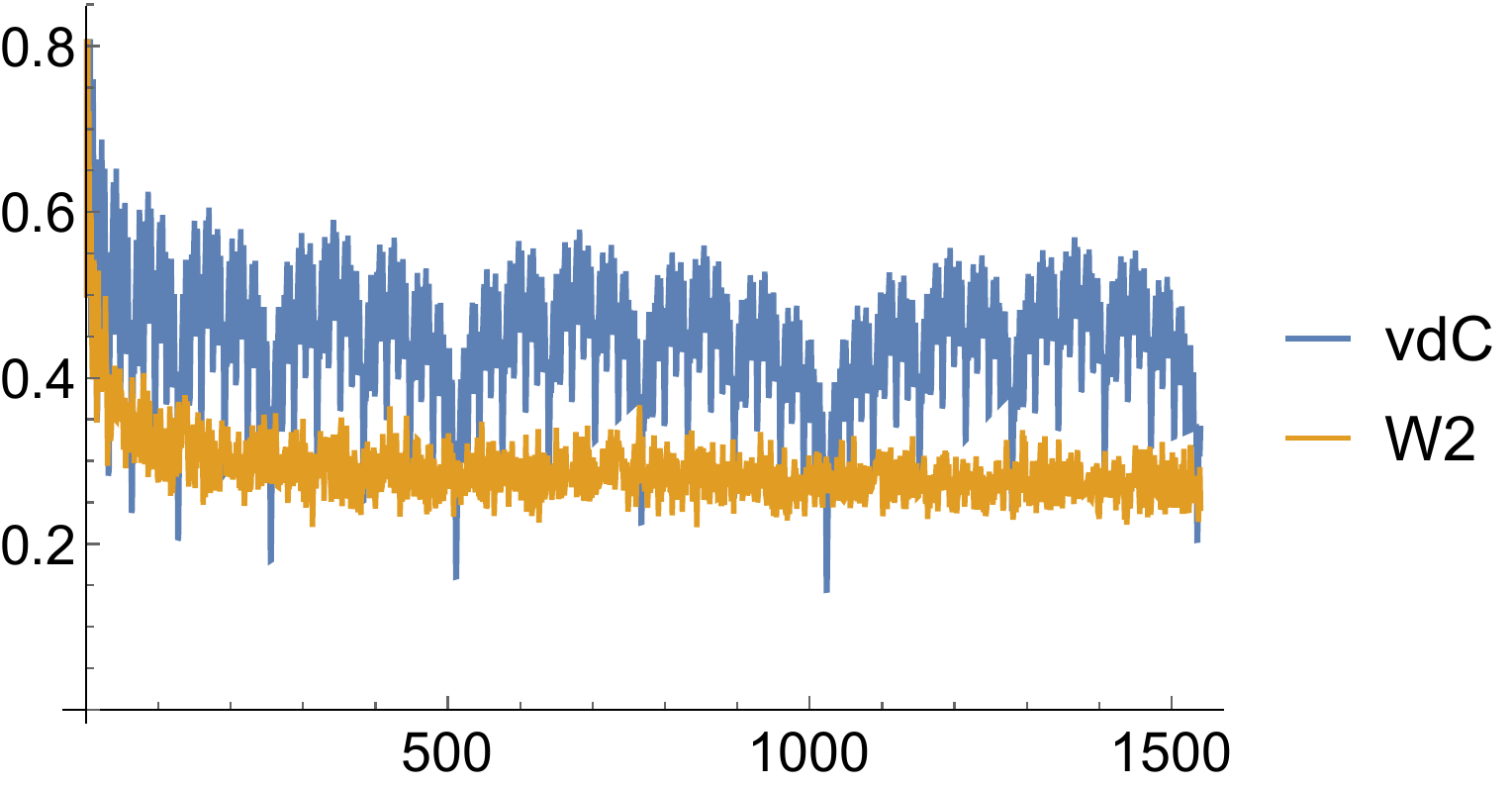}};
\node at (6.5,0) {\includegraphics[width=0.5\textwidth]{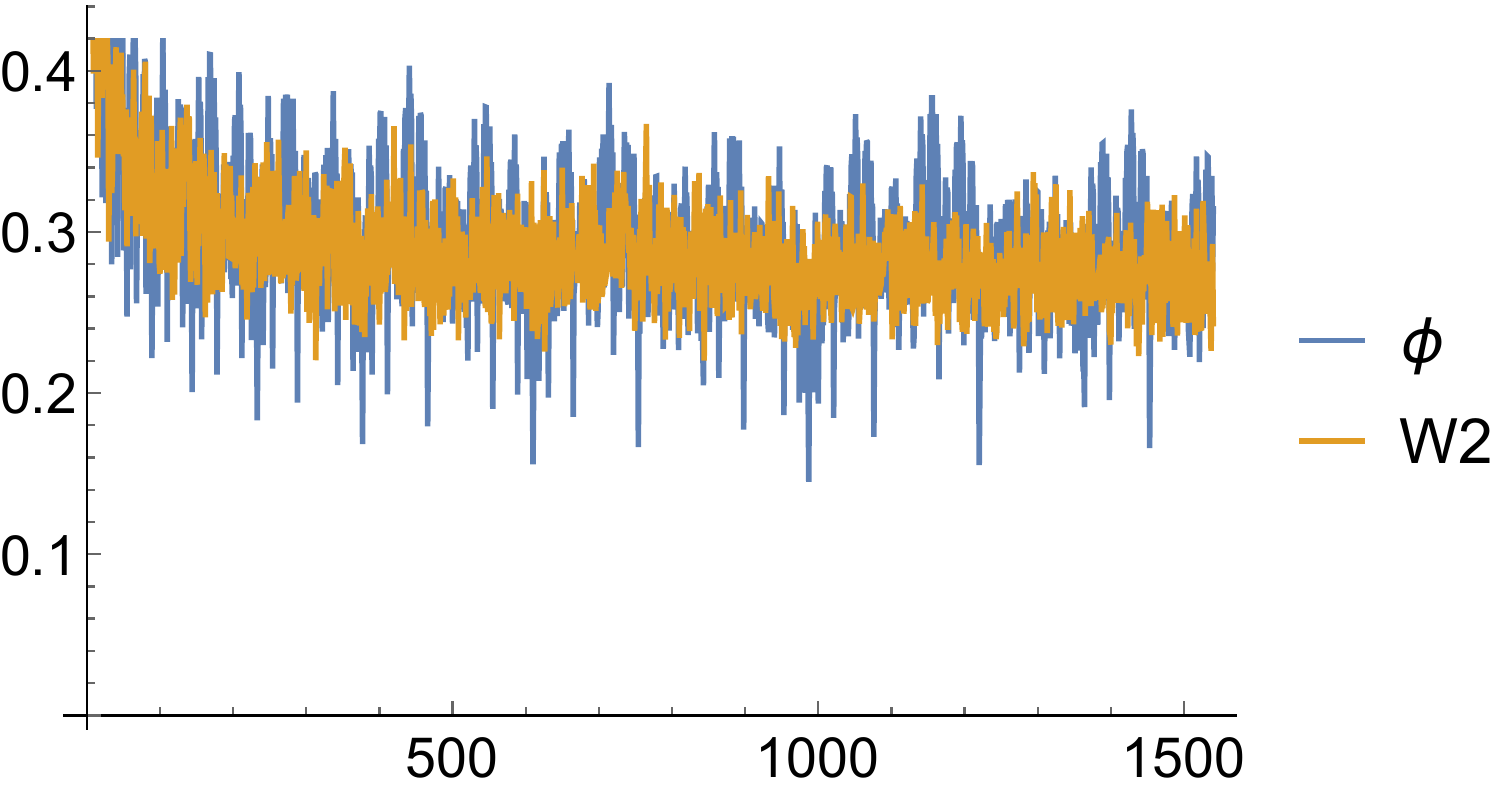}};
\end{tikzpicture}
\caption{Comparison to van der Corput (left) and $n \phi \mod 1$ (right).}
\end{figure}
\end{center}

We compared the maximum deviation (multiplied with $1/\log{n}$ to rescale the leading asymptotic order) between the Kritzinger sequence started with $x_1 = 1/2$, the van
der Corput sequence in base 2 and the Kronecker sequence with $\alpha = (1+\sqrt{5})/2$. The Kritzinger sequence is somewhat superior to the van der Corput sequence and very comparable to the golden ratio Kronecker sequence (albeit it seems to have less variance).
We emphasize that these two sequences are among the most regular sequences known and their behavior
is within a small factor from the theoretical limit: indeed, irrational rotations on the torus are among the most well studied object in many different
fields of mathematics while the van der Corput sequence was specifically designed to be as regular as possible. The Kritzinger
sequence appears to draws its regularity from the geometry of the Wasserstein distance.

\subsection{Improved Regularity}
Our main result is a new regularity statement for Kritzinger sequences. Such a sequence is assumed to be initialized with an arbitrary finite set of real numbers and then extended in the greedy fashion.

\begin{theorem} Every Kritzinger sequence $(x_k)_{k=1}^{\infty}$ has infinitely
many  $n \in \mathbb{N}$ with
$$ \max_{0 \leq x \leq 1} \left| \int_0^x  \# \left\{1 \leq k \leq n: x_k \leq y\right\} -  n y  ~dy \right| \leq 2 \cdot n^{1/3}.$$
\end{theorem}

In light of Fig. 2, we naturally expect that $n^{1/3}$ can be replaced by $\mathcal{O}(\log{n})$. The scaling of the $L^1-$discrepancy
suggests that $\mathcal{O}(\sqrt{\log{n}})$ may be a plausible guess. It is not clear to us whether it could be as small as $\mathcal{O}(1)$.
Theorem 2 breaks the square root barrier, in particular, it shows in a concrete way that Kritzinger sequences
are more regular than i.i.d. random variables (for which this quantity is $\sim n^{1/2}$ w.h.p).

\subsection{A curious inequality.}
A new technical ingredient of the proof is a curious new inequality for real-valued functions.
\begin{lemma}[Main Lemma] Let $g:[0,1] \rightarrow \mathbb{R}$ be continuous and $g \not\equiv 0$. Then 
$$\max_{0 \leq z \leq 1} ~  \int_0^z g(x) x ~dx -  \int_z^1 g(x) (1-x) ~dx \geq  \frac{1}{16} \frac{1}{ \| g\|_{L^2}^2} \left(\max_{0 \leq x \leq 1} \left| \int_0^x g(y) dy\right|\right)^3.$$
\end{lemma}
This inequality is sharp up to at most a factor of 8: take $g$ to be the characteristic function on $[0, \varepsilon]$. Then
$$ \max_{0 \leq z \leq 1} ~  \int_0^z g(x) x ~dx -  \int_z^1 g(x) (1-x) ~dx = \int_0^{\varepsilon} g(x) x dx = \frac{\varepsilon^2}{2}$$
while
$$  \frac{1}{16}\frac{1}{ \| g\|_{L^2}^2} \left(\max_{0 \leq x \leq 1} \left| \int_0^x g(y) dy\right|\right)^3 = \frac{1}{16 \varepsilon} \cdot \varepsilon^3 = \frac{\varepsilon^2}{16}.$$
This inequality plays a role in a part of the proof that is somewhat modular and could relatively easily be exchanged by something else. Any lower bound 
$$\max_{0 \leq z \leq 1} ~  \int_0^z g(x) x ~dx -  \int_z^1 g(x) (1-x) ~dx \geq  \mbox{informative quantity}$$
could be used to derive a regularity statement for Kritzinger sequences.

\subsection{Related Results.} The problem of irregularities of distributions in dimensions $d \geq 2$ has proven to be challenging \cite{beck, bc, bil3, bil4}. Indeed, at this point there is not even a clear consensus on what the optimal scale of regularity should be. Most examples of highly regular sequences seem to follow the spirit of van der Corput sequences or generalized Kronecker sequences (we refer to the textbooks \cite{chazelle, dick, drmota, kuipers}). This naturally suggests investigating the possibility of new constructions. Such constructions were given by the author in \cite{stein, steini} and by Brown and the author in \cite{brown, brown2}. Pausinger \cite{pausinger} showed that a large class of greedy constructions can replicate the van der Corput sequence: this means that, at least for some suitable initial conditions, they can produce sequences at the theoretical limit of regularity. One of the most interesting example of a greedy sequence was recently given by Ralph Kritzinger \cite{kritzinger}: the arising sequence is, empirically, as regular as any of the other greedy sequences but additionally consists of rational numbers which greatly simplifies computation.
Most of these greedy constructions can easily be shown to have regularity $\lesssim n^{-1/2}$, a barrier which has resisted improvement. A notable exception is \cite{steini2} where additional structure in the complex plane could be used to show that a classic greedy constructions for polynomials has optimal rate up to possibly logarithmic factors (see also work of Beck \cite{beck0} resolving a question of 
Erd\H{o}s \cite{erd}). The Wasserstein distance has been introduced to the study of irregularities of distributions in \cite{stein3} and was then further developed in \cite{brown}. We especially emphasize work of C. Graham \cite{graham} who established an irregularities of distributions phenomenon for $d=1$. An explicit inequality relating the $W_2$ distances to the Green Energy \cite{stein4} has lead to the following rather general principle (established by Brown and the author in \cite{brown}): on compact manifolds in dimension $d \geq 3$ normalized to have volume 1, there always exist greedy sequences $(x_k)_{k=1}^{\infty}$ obtained by minimizing the Green energy for which
$ W_2\left(\frac{1}{n} \sum_{k=1}^{n} \delta_{a_k}, dx \right) \lesssim n^{-1/d}$ uniformly in $n$.

\section{Proofs}
 \S 3.1 gives two proofs of Theorem 1. \S 2.2 derives a useful alternative representation, \S 2.3 proves an Lemma which is enough to prove a regularity rate of $\lesssim n^{1/2}$. \S 2.4 proves the Main Lemma which is then used in \S 2.5 to prove Theorem 2.
\subsection{Proof of Theorem 1}
\begin{proof}
Let $0 \leq x_1 \leq x_2 \leq \dots \leq x_n \leq 1$ be given. Observe that
\begin{align*}
 W_2\left( \frac{1}{n} \sum_{k=1}^{n} \delta_{x_k}, dx \right)^2&= \sum_{k=1}^{n} \int_{ \frac{k-1}{n}}^{\frac{k}{n}} (x_k - z)^2 dz \\
 &= \sum_{k=1}^{n}  \int_{ \frac{k-1}{n}}^{\frac{k}{n}} x_k^2 - 2 x_k z + z^2 dz \\
 &= \frac{1}{3} + \frac{1}{n} \sum_{k=1}^n x_k^2 - \sum_{k=1}^{n} x_k \frac{2k-1}{n^2}.
\end{align*}
Rescaling by a factor of $n^2$, we have
$$ n^2 \cdot W_2\left( \frac{1}{n} \sum_{k=1}^{n} \delta_{x_k}, dx \right)^2 =  \frac{n^2}{3} +  n\sum_{k=1}^n x_k^2 - \sum_{k=1}^{n} x_k (2k-1).$$
Let us now consider the original set and let us add an additional point $x$. We introduce the variables such that
$$ \left\{x_1, \dots, x_n, x \right\} = \left\{y_1, \dots, y_{n+1} \right\}$$
and $y_1 \leq y_2 \leq \dots \leq y_{n+1}$. A computation shows that
\begin{align*}
(n+1)^2 \cdot W_2^2\left( \frac{1}{n+1} \sum_{k=1}^{n+1} \delta_{y_k}, dx \right) &=  \frac{(n+1)^2}{3} +  (n+1)\sum_{k=1}^{n+1} y_k^2 - \sum_{k=1}^{n+1} y_k (2k-1)\\
&=n^2 \cdot W_2^2\left( \frac{1}{n} \sum_{k=1}^{n} \delta_{x_k}, dx \right) + \frac{(n+1)^2 - n^2}{3} \\
&+ (n+1) x^2 + \sum_{k=1}^{n} x_k^2 \\
&+\sum_{k=1}^{n} x_k \cdot (2k-1)  - \sum_{k=1}^{n+1} y_k \cdot (2k-1).
\end{align*}
Among these terms only two, $(n+1)x^2$ and the last sum, actually depend on $x$. 
The next step consists in simplifying the difference between the last two sums to exploit some additional cancellation. 
There are now three cases: (1) $x< x_1$, (2) $x> x_n$ or (3) there exists  $j$ such that
$$ x_j \leq x < x_{j+1}.$$
We will only deal with case (3), the other two cases can be dealt with analogously.
Under this assumption, we necessarily have
$$ y_i = \begin{cases} x_i \quad &\mbox{if}~ i \leq j\\
x \qquad &\mbox{if}~i = j+1 \\
x_{i-1} \qquad &\mbox{if}~i \geq j+2 \end{cases}$$
Thus we have
\begin{align*}
 \sum_{k=1}^{n+1} y_k (2k-1) - \sum_{k=1}^{n} x_k (2k -1) &= (2j + 1) x + 2 \sum_{k \geq j+1} x_k\\
 &= x + 2 \sum_{k=1}^{n} \max\left\{x, x_k \right\}.
\end{align*}
Collecting all the terms that actually depend on $x$, our goal is to minimize
$$F(x) = (n+1)x^2 -x - 2 \sum_{i=1}^{n} \max\left\{x, x_i \right\}.$$
which is exactly the Kritzinger functional.
\end{proof}
Kritzinger's functional can be derived from trying to minimize the $L^2-$discrepancy
$$ \int_0^1 \left| \# \left\{1 \leq k \leq n: x_k \leq x\right\} - nx \right|^2 dx.$$
In hindsight, the fact that these functionals are related is perhaps not too surprising: the fact that Wasserstein distances can be equivalently defined over empirical cumulative distribution functions is a feature of us working
in one dimension (see, for example, Vallender \cite{vall}). Another quick proof of Theorem 1 could be given by appealing to the formula (see e.g. \cite[Equation 10]{kritzinger}) that for ordered points
$$ \int_0^1 \left| \# \left\{1 \leq k \leq n: x_k \leq x\right\} - nx \right|^2 dx  = n \sum_{k=1}^{n} \left(x_k - \frac{2k-1}{2n} \right)^2 + \frac{1}{12}$$
while
$$ n^2 \cdot W_2\left( \frac{1}{n} \sum_{k=1}^{n} \delta_{x_k}, dx \right)^2 =  \frac{n^2}{3} +  n\sum_{k=1}^n x_k^2 - \sum_{k=1}^{n} x_k (2k-1).$$
This quickly implies that
$$ \int_0^1 \left| \# \left\{1 \leq k \leq n: x_k \leq x\right\} - nx \right|^2 dx - n^2 \cdot W_2^2\left( \frac{1}{n} \sum_{k=1}^{n} \delta_{x_k}, dx \right) = c_n$$
for a constant $c_n$ that is independent of the set of points. Therefore, minimizing one of the functionals is equivalent to minimizing the other.
The explicit form of $F$ allows to prove that minimizers are rational numbers.
\begin{proposition}[Kritzinger \cite{kritzinger}] Let $0 \leq x_1 \leq x_2 \leq \dots \leq x_n \leq 1$ and let 
$$F(x) = (n+1)x^2 -x - 2 \sum_{k=1}^{n} \max\left\{x, x_k \right\}.$$
Then the minimizer is different from all the existing points and of the form $x = j/(2n)$ for for some odd
integer $j \in 2\mathbb{Z} + 1$.
\end{proposition}
\begin{proof} Note that the function $F$ is continuous, not differentiable in the points $x_1, \dots, x_n$ but differentiable everywhere else. In particular, if $x$ is different from the $n$ points, then
$$ F'(x) = 2 (n+1) x - 1 - 2 \cdot \# \left\{ 1 \leq k \leq n: x_k < x \right\}.$$
The only candidates for roots outside of the existing points are numbers of the form
$$ x = \frac{1 + 2\# \left\{ 1 \leq k \leq n: x_k < x \right\}}{2n+2}$$
which is exactly of the desired form. However, we also note that $F'$ is piecewise linear between the points. If it were now the case that
$F$ assumes its global minimum in $x_k$, then we would have to have $F'(x_k - \varepsilon) \leq 0$ and $F'(x_k + \varepsilon) \geq 0$ for sufficiently
small values of $\varepsilon$ which is ruled out by the explicit form of $F'$.
\end{proof}

\subsection{Some Preparatory Work.} Since we now know that minimizing the $W_2$ cost is equivalent to minimizing the $L^2-$discrepancy, we will instead work with this new, renormalized functional.
To this end we introduce, for any given set of $n$ points in the unit interval, the counting function
$$ f_n(x) =  \# \left\{ 1 \leq k \leq n: x_k \leq x \right\}.$$
We will also use, throughout the rest of the paper, the rescaled version
$$ g_n(x) =  \# \left\{ 1 \leq k \leq n: x_k \leq x \right\} - nx.$$

Assume $x_{n+1}$ is being added.
A straightforward computation shows that
\begin{align*}
\int_0^1 ( f_{n+1}(x) - (n+1)x)^2 dx &= \int_0^{x_{n+1}} ((f_n(x) - nx) -x)^2 dx\\
& + \int_{x_{n+1}}^1 ((f_n(x)-nx) +(1 -x))^2 dx \\
& = \int_0^{1} (f_n(x) - nx)^2 dx - 2 \int_0^{x_{n+1}} (f_n(x) - nx)x ~dx \\
&+ 2 \int_{x_{n+1}}^{1} (f_n(x) - nx)(1-x) dx+ \frac{x_{n+1}^3}{3}  + \frac{(1-x_{n+1})^3}{3}.
\end{align*}
Note that some of these terms do not depend on $x_{n+1}$ while others do. Collecting only those
terms that the depend on $x_{n+1}$ shows that one needs to solve
$$  E(z)  + \frac{z_{}^3 + (1-z)^3}{3}  \rightarrow \min,$$
where the function $E$ is defined as
\begin{align*}
 E(z) = -2 \int_0^{z} (f_n(x) - nx) x ~dx +  2 \int_{z_{}}^{1} (f_n(x) - nx)(1-x) ~dx.
 \end{align*}
For the remainder of the paper, we will work with this representation.
\subsection{An Elementary Lemma} The purpose of this section is to introduce a basic Lemma and to use it to prove the missing part of Kritzinger's Theorem: that
$$  \int_0^{1} (f_n(x) - nx)^2 dx \leq \frac{n}{3} + c,$$
where $c>0$ only depends on the initial configuration. Subsequent arguments will then both use and refine this step.
\begin{lemma}[Basic Form] Let $g:[0,1] \rightarrow \mathbb{R}$ be continuous. Then
$$\max_{0 \leq z \leq 1} ~  \int_0^z g(x) x ~dx -  \int_z^1 g(x) (1-x) ~dx \geq 0.$$
\end{lemma}
\begin{proof}
Using integration by parts  
\begin{align*}
 \int_0^z g(x) x ~dx -  \int_z^1 g(x) (1-x) ~dx &=   \int_0^z g(x) x ~dx +  \int_z^1 g(x) (x-1) ~dx \\
 &= \int_0^1 g(x) x dx - \int_z^1 g(x) dx.
 \end{align*}
Introducing the anti-derivative $G: [0,1] \rightarrow \mathbb{R}$ via
$ G(s) = \int_0^s g(x) dx$
we have, integrating again by parts,
\begin{align*}
\int_0^1 g(x)x dx = G(x) x \big|_0^1 - \int_0^1 G(x) dx= G(1) - \int_0^1 G(x) dx
\end{align*}
as well as
$  \int_z^1 g(x) dx = G(1) - G(z).$
Thus
\begin{align*}
 \int_0^z g(x) x ~dx -  \int_z^1 g(x) (1-x) ~dx &=   G(z) - \int_0^1 G(x) dx.
 \end{align*}
Every function has to be at least as large as its average somewhere and
$$  \max_{0 \leq z \leq 1} G(z) - \int_0^1 G(x) dx \geq 0.$$
\end{proof}

One sees that the proof is merely a sequence of identities: the part where one can hope to gain is the
last step which is clearly only sharp whenever $G$ is constant.
The elementary Lemma now implies the missing part of Corollary 1.

\begin{proposition} We have 
$$  \int_0^1 ( f_{n+1}(x) - (n+1)x)^2 dx \leq \int_0^1 ( f_{n}(x) - n x)^2 dx + \frac{1}{3}.$$
\end{proposition}
\begin{proof} Note that we have
\begin{align*}
 \int_0^1 ( f_{n+1}(x) - (n+1)x)^2 dx &= \int_0^1 ( f_{n}(x) - n x)^2 dx - 2 \cdot E(x_{n+1})\\
 &+  \frac{x_{n+1}^3 + (1-x_{n+1})^3}{3}.
\end{align*}
There exists $x_{n+1}$ such that $E(x_{n+1}) \geq 0$ from which the result follows since
$$ \forall~ 0\leq z \leq 1 \qquad \frac{z^3 + (1-z)^3}{3} \leq \frac{1}{3}$$
\end{proof}

\subsection{A refined Lemma.} The purpose of this section is to establish what is perhaps the core of the improvement: instead of merely using that a function has to exceed its average, as in the proof of the basic Lemma, we quantify the gain from below.
\begin{lemma}[Main Lemma] Let $g:[0,1] \rightarrow \mathbb{R}$ be continuous and $g \not\equiv 0$. Then 
$$\max_{0 \leq z \leq 1} ~  \int_0^z g(x) x ~dx -  \int_z^1 g(x) (1-x) ~dx \geq  \frac{1}{16} \frac{1}{ \| g\|_{L^2}^2} \left(\max_{0 \leq x \leq 1} \left| \int_0^x g(y) dy\right|\right)^3.$$
\end{lemma}
\begin{proof} Using again the antiderivative
$$ G(x) = \int_0^x g(y) dy$$
the computations in the basic Lemma already established that
$$\max_{0 \leq z \leq 1} ~  \int_0^z g(x) x ~dx -  \int_z^1 g(x) (1-x) ~dx = \max_{0 \leq z \leq 1} G(z) - \int_0^1 G(x) dx \geq 0.$$
We will work with this expression instead. In order to emphasize that the integral in this expression is truly a constant,
the average value, we will abbreviate 
$$ \overline{G} = \int_0^1 G(x) dx \in \mathbb{R}$$
We start by arguing that if $G$ assumes a value much smaller than the average, then it also has to assume values that are larger than the average. The difficulty here is that `much smaller than the average' is a pointwise statement which we have to turn into a statement about the integral. This is summarized in the following statement.\\

\textbf{Fact.} \textit{We have}
$$\max_{0 \leq z \leq 1} G(z) - \overline{G} \geq \frac{1}{8 \|g\|_{L^2}^2} \cdot \left( \overline{G} - \min_{0 \leq x \leq 1} G(x)  \right)^3.$$

\begin{proof}[Proof of the Fact.]
For simplicity of exposition,we abbreviate 
$$ M = \overline{G} - \min_{0 \leq x \leq 1} G(x) \geq 0.$$
Let us
assume the minimum is assumed in 
$$ G(z_1) = \min_{0 \leq x \leq 1} G(x).$$
 Let $J$ be the shortest interval containing both $z_1$ and a point $0 \leq z_2 \leq 1$ for which
$$ G(z_2) = G(z_1) + \frac{M}{2}.$$
The existence of such a point $z_2$ follows from the fact that $G$ is continuous and
$$ G(z_1) \leq G(z_1) + \frac{M}{2} \leq G(z_1) + M = \overline{G}$$
and the function $G$ assumes both the value $G(z_1)$ and the value $\overline{G}$ and the intermediate value theorem applies.
If $M = 0$ there is nothing to prove, thus we can assume $M \neq 0$ implying $z_1 \neq z_2$.  The shortest interval $J$ containing both
$z_1$ has positive length and $z_1, z_2$ are its endpoints. 
Using Cauchy-Schwarz on $J$,
\begin{align*}
\frac{M}{2} &= \left| G(z_2) - G(z_1) \right| = \left| \int_{J}^{} g(x) dx \right|\\
& \leq |J|^{1/2} \left( \int_{J} g(x)^2 dx \right)^{1/2} \leq  |J|^{1/2} \left( \int_{0}^1 g(x)^2 dx \right)^{1/2} = |J|^{1/2} \cdot \|g\|_{L^2}.
\end{align*}
This provides an upper bound on $M$.  Our next goal will be to show that $|J|$ cannot be too long: if it is long, then $\max_{0 \leq z \leq 1} G(z) - \overline{G} $ has to be large which is exactly what we are trying to show. This follows at once from
\begin{align*}
\overline{G} &= \int_0^1 G(x) dx = \int_{J} G(x) dx + \int_{[0,1] \setminus J} G(x) dx \\
&\leq \left(\overline{G} - \frac{M}{2}\right) \cdot |J| + (1-|J|) \cdot \left(\overline{G} + \left[\max_{0 \leq z \leq 1} G(z) - \overline{G} \right] \right)\\
& \leq \overline{G} - \frac{M}{2} \cdot |J| + \left[\max_{0 \leq z \leq 1} G(z) - \overline{G} \right]
\end{align*}
from which we infer that
$$ \max_{0 \leq z \leq 1} G(z) - \overline{G}  \geq |J| \cdot \frac{M}{2}.$$
 We can now combine these bounds. We have
$$ \frac{M^2}{4 \|g\|_{L^2}^2} \leq |J| \leq \frac{2}{M} \cdot \left[\max_{0 \leq z \leq 1} G(z) - \overline{G} \right]$$
which implies that
$$ \max_{0 \leq z \leq 1} G(z) - \overline{G}  \geq \frac{M^3}{8 \|g\|_{L^2}^2}.$$
\end{proof}
The fact will now allow us to conclude the proof of the Main Lemma.
Consider the values assumed by $G$. Since $G$ is continuous, we have
$$ \left\{G(x): 0 \leq x \leq 1 \right\} = [A,B],$$
where $A \leq \overline{G} \leq B$. We introduce $Y = B-A$. As a first observation, we have
$$  \max_{0 \leq x \leq 1} \left| \int_0^x g(y) dy\right|  = \max_{0\leq x \leq 1} |G(x)|. $$
Note that $G(0) = 0$ and thus $A \leq 0 \leq B$ from which we deduce that
$$  \frac{Y}{2} \leq \max_{0 \leq x \leq 1} \left| \int_0^x g(y) dy\right|  = \max_{0\leq x \leq 1} |G(x)| \leq Y.$$
We will now show that a large value of $Y$ will imply that $G$ has to exceed its maximum by at least a little bit (depending on $Y$).
For this, we write $Y = Y_1 + Y_2$, where $Y_1 = B- \overline{G}$ and $Y_2 = \overline{G} - A$.
Then, tautologically, 
$$\max_{0 \leq z \leq 1} G(z) - \overline{G} = Y_1.$$
Simultaneously, appealing to the fact proved above, we have
$$\max_{0 \leq z \leq 1} G(z) - \overline{G} \geq \frac{Y_2^3}{8 \| g\|_{L^2}^2}.$$
Altogether, we have
$$\max_{0 \leq z \leq 1} G(z) - \overline{G} \geq \max\left\{ Y_1,  \frac{Y_2^3}{8 \| g\|_{L^2}^2}\right\} = \max\left\{ Y_1,  \frac{(Y-Y_1)^3}{8 \| g\|_{L^2}^2}\right\} $$
which we turn into a smoother object via
$$\max_{0 \leq z \leq 1} G(z) - \overline{G} \geq \frac{1}{2} \left( Y_1 +  \frac{(Y-Y_1)^3}{8 \| g\|_{L^2}^2}\right).$$
Differentiation shows that
$$ \frac{\partial}{\partial Y_1} \left[Y_1 +  \frac{(Y-Y_1)^3}{8 \| g\|_{L^2}^2} \right]= 1 + \frac{3}{8} \frac{(Y-Y_1)^2}{ \| g\|_{L^2}^2} \geq 1$$
which shows that the minimum is assumed for $Y_1 = 0$. Thus
$$ \max_{0 \leq z \leq 1} G(z) - \overline{G} \geq \frac{1}{16} \frac{Y^3}{ \| g\|_{L^2}^2}.$$
Altogether, we obtain, as desired,
$$ \max_{0 \leq z \leq 1} G(z) - \overline{G} \geq \frac{1}{16} \frac{1}{ \| g\|_{L^2}^2} \left(\max_{0 \leq x \leq 1} \left| \int_0^x g(y) dy\right|\right)^3.$$
\end{proof}

\subsection{Proof of Theorem 2}
\begin{proof}
Suppose now that $(x_k)_{k=1}^{\infty}$ is a Kritzinger sequence. 
We note that 
$$  \int_0^{1} (f_n(x) - nx)^2 dx = \int_0^{1} g_n(x)^2 dx \leq \frac{n}{3} + c,$$
 where $c$ only depends on the initial configuration. Suppose now there exists $N \in \mathbb{N}$ such that for all $n \geq N$ we have
 $$ \max_{0 \leq x \leq 1}  \left| \int_0^x g_n(y) dy\right| \geq 2 \cdot n^{1/3}.$$
We will use this to deduce a contradiction. We have
\begin{align*}  \int_0^{1} g_{n+1}(x)^2 dx &\leq  \int_0^{1} g_n(x)^2 dx + \frac{1}{3}  \\
&- \max_{0 \leq z \leq 1} 2\left(   \int_0^{z} g_n(x) x ~dx - \int_{z_{}}^{1} g_n(x) (1-x) ~dx\right).
\end{align*}
The Main Lemma implies
\begin{align*}
  \max_{0 \leq z \leq 1} 2\left(   \int_0^{z} g_n(x) x ~dx - \int_{z_{}}^{1} g_n(x) (1-x) ~dx\right) &\geq \frac{1}{8} \frac{1}{ \| g_n\|_{L^2}^2} \left(\max_{0 \leq x \leq 1} \left| \int_0^x g_n(y) dy\right|\right)^3 \\
  &\geq \frac{1}{8} \frac{1}{\frac{n}{3} + c} 8 n = \frac{n}{n/3 + c}.
  \end{align*}
However, for $n$ sufficiently large (say $n \geq N_1$ where $N_1$ depends only on $c$), this implies, for all $n \geq \max(N, N_1)$
$$  \int_0^{1} g_{n+1}(x)^2 dx  \leq  \int_0^{1} g_{n}(x)^2 dx - 2$$
which leads to a contradiction since these integrals are always positive.
\end{proof}

\textbf{Remark.} We note that this allows for a small refinement: the argument shows that for all $N$ sufficiently large, there always exist $N \leq n \leq 100N$ for which
 $$ \max_{0 \leq x \leq 1}  \left| \int_0^x g_n(y) dy\right| \leq 2 \cdot n^{1/3}.$$

\end{document}